\documentstyle{lms}

\newtheorem{thm}{Theorem}

\newtheorem{lem}{Lemma} 
\newnumbered{ex}{Example}
\newunnumbered{dfn}{Definition}
\newunnumbered{rem}{Remark}

\extraline{Research of latter funded by EPSRC studentship 94007848.}
\classno{20F32}

\begin{document}

\title[The $p$-modular Descent Algebra of the Symmetric Group]{The p-modular
Descent Algebra of the Symmetric Group}  
\author{M. D. Atkinson \and\ S. J. van Willigenburg} 
\maketitle
 
\begin{abstract} The descent algebra of the 
symmetric group, over a field of non-zero characteristic $p$, is studied. A
homomorphism into the algebra of generalised p-modular characters of the 
symmetric group is defined.
This is then used to determine the radical, and its nilpotency index. It 
also allows the irreducible representations of the descent algebra to be 
described.
\end{abstract}

\section{Introduction}

In 1976, Louis Solomon defined a family of algebras associated with Coxeter
groups \cite{Solomon}.  In the case of symmetric groups their 
definition can be expressed as
follows:

If $\sigma$ is any permutation in the symmetric group $S_{n}$ written in image form 
(e.g. [1342])
then the {\em signature} of $\sigma$ is the sequence of signs $\{ x_{i}\}_{i=1}^{n-1}$
where 
\[ x_{i}=\left\{
\begin{array}{cc} +&\mbox{ if }(i+1)^{\sigma}-i^{\sigma}>0\\
-&\mbox{ if }(i+1)^{\sigma}-i^{\sigma}<0 
\end{array}
\right. 
\] 
For example, [1342] has
the signature $\varepsilon = [++-]$. Such signatures partition the $n!$
permutations of $S_{n}$ into $2^{n-1}$ disjoint signature classes, and 
we denote the sum of
all elements in a given signature class, $\varepsilon$, by 
$A_{\varepsilon}$.  Solomon proved that, for any two signatures $\varepsilon$, 
$\eta$, $A_{\varepsilon}A_{\eta}$ is a linear
combination (with non-negative integer coefficients) of signature class sums.
Hence the signature class sums span a sub-algebra of the group algebra of
dimension $2^{n-1}$ which has become known as the {\em descent algebra} $\Sigma_{n}$
\cite{Gandr}.

The algebra $\Sigma_{n}$ is not semi-simple.  Indeed, Solomon proved that 
the dimension of its radical is $2^{n-1}-p(n)$ (where $p(n)$ is the 
partition function). Garsia and Reutenauer,
 in their extensive paper \cite{Gandr}, gave another proof of this 
 result; they also derived other natural bases for $\Sigma_{n}$ and 
 determined the 
Cartan invariants. In other work on $\Sigma_{n}$, Atkinson \cite{Atkinson} defined a family of 
homomorphisms on $\Sigma_{n}$, including an epimorphism from $\Sigma_{n}$ 
to \(\Sigma_{n-1}\), and proved that the 
nilpotency index of the radical is $n-1$; and very recently Gelfand {\em et al} 
\cite{Gelfand}
have used the descent algebra in a key way in their work on non-symmetric 
functions. 
  In all these papers, $\Sigma_{n}$ has
been studied as an algebra over a field of characteristic zero.  However, 
since the structure constants of the algebra are integers, it is also 
possible to define the descent algebra over fields \( {\cal F}_p \) of 
any prime order 
$p$.  For values of $p>n$ all the above results extend 
virtually unchanged but, as we shall see in this paper, $p\leq n$ gives 
rise to a more complicated situation.  In this case the dimension of the 
radical depends on $p$ as well as $n$.  Nevertheless we are able to
 identify the radical (by 
giving a natural basis for it), determine its nilpotency index, and 
describe the irreducible representations of the descent algebra.

It is convenient to work with the alternative definition of $\Sigma_{n}$ 
given below (and justified in  \cite{Gandr}) in which $\Sigma_{n}$ is defined 
by a basis $\{B_q\}$ indexed by compositions of $n$.

If $q = [a_{1},a_{2}, \ldots ,a_{s}]$ and $r = [b_{1},b_{2}, \ldots ,b_{t}]$
are compositions of $n$ we define $S(q,r)$ to be the set of
all $s \times t$ matrices $Z = (z_{ij})$ with non-negative integer entries 
such that
\begin{enumerate}
\item $\sum_{j} z_{ij} = a_{i}$ for each $i = 1,2, \ldots ,s$
\item $\sum_{i} z_{ij} = b_{j}$ for each $j = 1,2, \ldots ,t$
\end{enumerate}

Multiplication in $\Sigma_{n}$ is then defined by the rule
\begin{equation}
B_{q}B_{r} = \sum_{Z \in S(q,r)} B_{[z_{11},z_{12}, \ldots
,z_{1s},z_{21}, \ldots ,z_{2s}, \ldots ,z_{r1}, \ldots ,z_{rs}]}
\label{Bm-ult}
\end{equation}
\label{de-scentalgdef}

\begin{rem}
Due to some $z_{ij}$ possibly being zero, 
\[[z_{11},z_{12}, \ldots ,z_{1s},z_{21}, \ldots ,z_{2s}, \ldots ,z_{r1},
\ldots ,z_{rs}]\]
may not be a composition, but it can be identified with the composition
obtained by omitting zero components and, because of this, the multiplicity of
a basis element $B_{s}$ in the right hand side of Equation ~\ref{Bm-ult} may be
greater than one.
\end{rem}

\begin{ex}
\label{fi-rstmult}
If $n = 4 ,q = [2,2], r = [2,1,1]$ then $S(q,r)$ is the set of matrices
\[
\left(\begin{array}{ccc}
2&0&0\\
0&1&1\\
\end{array}\right)\quad
\left(\begin{array}{ccc}
0&1&1\\
2&0&0\\
\end{array}\right)\quad
\left(\begin{array}{ccc}
1&0&1\\
1&1&0\\
\end{array}\right)\quad
\left(\begin{array}{ccc}
1&1&0\\
1&0&1\\
\end{array}\right)
\]

Hence, \[B_{q}B_{r} = B_{[2,1,1]} + B_{[1,1,2]} + 2B_{[1,1,1,1]}\]
\end{ex}

In order to study the characteristic \(p\) analogue of $\Sigma_{n}$, we define 
${\cal Z}_n$ to be the subring of $\Sigma_{n}$ consisting of all integral 
combinations of the basis elements $\{B_q\}$, and consider its ideal 
${\cal P}_n=p{\cal Z}_n$.  We define $\Sigma(n,p)$ to be the 
quotient ring ${\cal Z}_n/{\cal P}_n$; $\Sigma(n,p)$ is clearly an 
algebra over \( {\cal F}_p \) which we term the \emph{p-modular descent 
algebra}.  Of course, $\Sigma(n,p)$ is the algebra that would arise if the 
field of coefficients in the definition of $\Sigma_{n}$ had been taken as 
\( {\cal F}_p \).

We let \( \rho_1:{\cal Z}_n\rightarrow \Sigma(n,p)\) be the natural 
homomorphism with kernel ${\cal P}_n$ and write 
\(\overline{B}_{q}=\rho_1(B_q)\).  The set \(\{\overline{B}_{q}\}\) is 
obviously a basis for $\Sigma(n,p)$ and, as already implied, the 
multiplication rule for \( \overline{B}_q\overline{B}_r \) is the same as 
for \( B_qB_r \) except that coefficients are reduced modulo $p$.  Thus, 
as a consequence of Example ~\ref{fi-rstmult}, in \( \Sigma(4,2) \), \( 
\overline{B}_{[2,2]}\overline{B}_{[2,1,1]}=\overline{B}_{[2,1,1]}+\overline{B}_{[1,1,2]} \).

Let $q=[a_{1},a_{2},\ldots ,a_{r}]$ be a composition of $n$, let
\( H_q=S_{a_1}\times S_{a_2}\times\ldots\times S_{a_r} \) be the 
corresponding Young subgroup of \( S_n \), let \( 1_q \) be the principal 
character of \( H_q \), and let \( \chi_q=1_q^{S_n} \) be the Young character 
corresponding to $q$.  Then the \( Z- \)module \( G_n \) consisting of 
all integral combinations of \( \{\chi_q\} \) is, by the Mackey formula, 
closed under pointwise product and so has a ring structure.
Solomon \cite{Solomon} proved that the linear map \( \theta:{\cal 
Z}_n\rightarrow G_n \) defined by \( \theta(B_q)=\chi_q \), for all 
compositions $q$, is a homomorphism of rings.  This map was a key tool in 
Solomon's paper; he proved that its kernel \( {\cal 
R}_n \) is 
spanned by all differences \( B_q-B_r \), where $q$ and $r$ induce the 
same partition of $n$, and that \( {\cal R}_n \) is nilpotent.

To extend these results to \( \Sigma(n,p) \) we let \( \rho_2 \) be the 
map defined on generalised characters in \( G_n \) (all of which have 
integral values) which simply reduces the character values modulo $p$, and 
we let \( G(n,p) \) denote the image of \( G_n \) under this map; 
clearly, \( G(n,p) \) is a commutative algebra over \( {\cal F}_p \).  The 
kernel of the composite map \[ {\cal Z}_n \rightarrow G_n \rightarrow 
G(n,p) \] obviously contains \( {\cal P}_n \) and so induces an epimomorphism 
of \( {\cal F}_p-\) algebras
\( \phi:\Sigma(n,p) \rightarrow G(n,p) \) which satisfies
\[ \phi(\rho_1(x))=\rho_2(\theta(x)) \mbox{ for all }x\in\Sigma_n\]

Writing \( \tilde{\chi_q} \) for \( \rho_2(\chi_q) \) we obtain, in 
particular, \( \phi(\overline{B}_q)=\tilde{\chi_q} \).  The homomorphism 
\( \phi \) will enable us to describe \( {\cal R}(n,p) \), the radical of \( 
\Sigma(n,p) \), in a manner similar to the description in \cite{Solomon} 
of the radical of \( \Sigma_n \).

We conclude this section by defining two binary relations on the set of 
compositions which we then use to describe some useful properties of 
the multiplication rule for \( B_qB_r \).

If $q$ and $r$ are compositions of $n$ which differ only in the order of 
their components then we write $q \approx r$.  The relation $q \approx r$ 
is an equivalence relation on the compositions of $n$ with, clearly, 
$p(n)$ equivalence classes.

There is also a partial order relation on the set of compositions.  We 
write $r \preceq q$ if the components of $q$ can be obtained
from the components of $r$ by repeatedly replacing adjacent components by
their sum.

\begin{dfn}
Two matrices are said to be {\em column equivalent}
 if one can be obtained from the other
by permuting the columns.
\end{dfn}

\begin{lem}
Let $B_{q}$ and $B_{r}$ be basis elements of $\Sigma_{n}$ and suppose 
that, in the composition $r$, the number of components equal to $i$ is 
denoted by $t_{i}$. Then 
\begin {enumerate}
\item If the coefficient of \( B_s \) in $B_qB_r$ is non-zero then \( 
s\preceq q \)
\item The coefficient of $B_q$ in the product $B_qB_r$
is a multiple of
\( t_{1}!t_{2}!\ldots t_{n}!\)
and this coefficient depends on the equivalence class of $r$ only
\item If $q \approx r$, the coefficient of $B_q$ in $B_qB_r$ 
is exactly $t_{1}!t_{2}!\ldots t_{n}!$
\end {enumerate}
\label{co-effbqbr} 
\end{lem}

\begin{proof}
The first statement follows from Lemma 1.1 of  \cite{Atkinson}.  To 
prove the remaining statements 
let \( q=[a_1,\ldots,a_u] \) and \( r=[b_1,\ldots,b_v] \).  A matrix \( Z 
\in S(q,r) \) which contributes to the coefficient of \( B_q \) in 
\(B_qB_r \) satisfies
\[\sum_{j} z_{ij} = a_{i} \mbox{ and } \sum_{i} z_{ij} = b_{j} \]
and the non-zero entries of the rows of $Z$, if read in serial order, 
yield \( a_1,\ldots, a_u \).  It follows that the $i^{\mbox{th}}$ row of 
$Z$ has a single non-zero entry which is equal to \( a_i \).  Note also 
that, since all 
\( b_j>0 \), every column of $Z$ has at least one non-zero entry.

The set of matrices \( {\cal Q} \) (if any) which satisfy these conditions 
falls into a number of column equivalence classes.  Each of these classes 
has precisely $t_{1}!t_{2}!\ldots t_{n}!$ members since the set of 
columns of one of the matrices in \( {\cal Q} \) with a common sum  may be 
permuted arbitrarily.  Thus the coefficient of \( B_q \) in 
\(B_qB_r \) is indeed a multiple of $t_{1}!t_{2}!\ldots t_{n}!$.  If $s$ 
is some composition equivalent to $r$ the set of matrices that is analogous 
 to \( {\cal Q} \) 
is related to \( {\cal Q} \) by permuting columns.  This proves the second statement.  For the third 
statement we note that, when $q \approx r$, \( {\cal Q} \) consists of 
exactly one column equivalence class since then the matrices will have 
exactly one non-zero entry in each column as well as each row.
\end{proof}
 
Note that the conclusions of Lemma~\ref{co-effbqbr} hold also for basis 
elements \( \overline{B}_q,\overline{B}_r \) of \( \Sigma(n,p) \) except 
that the coefficients in question must be reduced modulo $p$.

\section{The Form of the Radical and the Irreducible Representations of
 \( \Sigma(n,p) \)}

\begin{lem}
 \( G(n,p) \) has dimension  \( g(n,p) \) over \( {\cal F}_p \) where 
 \( g(n,p) \) is 
the number of  conjugacy classes of p-regular elements in \(S_n\).
\label{di-gnp}
\end{lem}

\begin{proof}
For each composition $q$ and partition $\pi$ let \( m_{q\pi} \) be the 
value of the character \(\tilde{\chi}_q \) on the conjugacy class of 
elements of \(S_n\) of cycle type $\pi$ and let $M$ be the 
\(2^{n-1}\times p(n)\) matrix \([m_{q\pi}]\).  Then \(\dim G(n,p)=\mbox{rank } 
M\).

If \( \pi_1,\pi_2\) are the partition cycle types of two elements of 
\(S_n\) with the same $p$-regular part then by \S 82 of \cite{Candr} the columns of 
$M$ which correspond to \( \pi_1,\pi_2\) are equal.  Thus \(\mbox{rank } 
M\leq g(n,p)\).

To prove that \(\mbox{rank } M\geq g(n,p)\) we list the rows of $M$ so that the 
first \(p(n)\) rows are indexed by a complete set of inequivalent 
compositions.  We can then consider the \( p(n)\times p(n)\) submatrix 
$N$ consisting of these rows and index them by partitions.  If the 
partitions indexing the rows and columns of $N$ are  listed 
lexicographically then $N$ is a lower triangular matrix; furthermore, if 
\(\pi=1^{t_1}2^{t_2}\ldots n^{t_n} \) is a typical partition then the 
\((\pi,\pi)\) diagonal entry of $N$ is \(t_1!t_2!\ldots t_n!\ \textup{mod}\ p\) 
(which follows from the tabloid method of evaluating permutation characters \cite
{Encyclomath}, p41).  By \cite{Encyclomath} p41 again there are \( g(n,p)
\) non-zero  diagonal entries and so \(\mbox{rank } M\geq g(n,p)\).\end{proof}

\begin{lem}
${\cal R}(n,p)) \subseteq \ker\phi$,  \label{ra-dinker}
\end{lem}

\begin{proof}
The image of \( \phi \) is a space of functions defined over a field and 
is therefore semi-simple. Consequently the two-sided nilpotent ideal \( 
\phi({\cal R}(n,p)) \) must be zero.

\end{proof}

\begin{thm}
$\Sigma(n,p)/{\cal R}(n,p)$ is commutative.
\label{co-mutt}
\end{thm}
 
\begin{proof}
Since \( {\cal R}_n \) is a nilpotent ideal of \( {\cal Z}_n \), 
\( \rho_1({\cal R}_n) \) is a nilpotent ideal of \( \Sigma(n,p) \), and 
therefore \( \rho_1({\cal R}_n)\subseteq{\cal R}(n,p) \).  Hence there exists 
an ideal \( {\cal S}_n \) of \( \Sigma_n \), the pre-image of \( {\cal 
R}(n,p) \), such that \( {\cal R}_n\subseteq {\cal S}_n \) and \( {\cal 
S}_n/{\cal P}_n \cong {\cal R}(n,p) \).  Since \( \Sigma(n,p) \cong 
{\cal Z}_n/{\cal P}_n \), \(\Sigma(n,p)/{\cal R}(n,p) \cong {\cal 
Z}_n/{\cal S}_n \) is a homomorphic image of \( {\cal Z}_n/{\cal R}_n 
\cong G_n \).  Since the latter ring is commutative the theorem follows.
\end{proof}

 \begin{lem}  Let $\overline{B}_{r}$ be a basis element of  
 $\Sigma(n,p)$.  Then $\overline{B}_{r}$
  is nilpotent if and only if $r$ has a component of multiplicity $p$ or 
  more.
 \end{lem}
 
 \begin{proof}
 Suppose that $r$ has \( t_i \) components equal to $i$.  Set 
 \[I=\langle \overline{B}_q | q \preceq r \rangle \]
 By Lemma~\ref{co-effbqbr}
  $I$ is a right ideal of  $\Sigma(n,p)$ and so right multiplication 
 by \( \overline{B}_r \) induces a linear transformation on $I$.  We 
 consider the matrix of this transformation with respect to the given 
 basis \( \overline{B}_{q_1}\ldots \overline{B}_{q_w} \) of I ordered so 
 that \( q_i \preceq q_j \) implies \( i \leq j \).  This matrix is, by 
 Lemma~\ref{co-effbqbr}, lower triangular with diagonal elements
  all equal to a multiple of \( 
 t_1!t_2!\ldots t_n!\ \textup{mod}\ p \).  Therefore the matrix is nilpotent if and 
 only if one of the multiplicities \( t_i \) is $p$ or more.  If the 
 matrix is not nilpotent then certainly $\overline{B}_{r}$ is not 
 nilpotent.  On the other hand, if the matrix is nilpotent then \( 
 I\overline{B}_r^t=0 \) for some $t$ and so, as \(\overline{B}_{r} \in I \), 
 \( \overline{B}_q^{t+1}=0 \).
 \end{proof}
 
 \begin{lem}
 $\dim {\cal R}(n,p)) \geq \dim\ker\phi$.
 \label{di-mr}
 \end{lem}
 
 \begin{proof}
 \( \Sigma(n,p)/{\cal R}(n,p) \) is a commutative semi-simple algebra and 
 so contains no non-zero nilpotent elements.  Hence all nilpotent 
 elements of \( \Sigma(n,p) \) are contained in \( {\cal R}(n,p) \).
 
 The elements \( B_q-B_r \) with \( q\approx r \) of \( \Sigma_n \) lie 
 in the radical of \( \Sigma_n \) (\cite{Solomon}, Theorem 3) and so 
 are all nilpotent.  Hence their images \( \overline{B}_q-\overline{B}_r 
 \) are also nilpotent; they span a subspace $U$ of \( {\cal R}(n,p) \) 
 of dimension \( 2^{n-1}-p(n) \).
 
 If $q$ is a composition with a component of multiplicity $p$ or more 
 then every composition $r$ with \( q\approx r \)  also has this 
 property.  We choose a complete set $A$ of inequivalent compositions 
 with this property; clearly the members of $A$ can be put in 1-1 
 correspondence with the set of {\em partitions\ } of $n$ which have a 
 part of multiplicity $p$ or more.  However, it is known that the number 
 of such partitions is the same as the number of partitions which have a 
 part divisible by $p$ \cite{Encyclomath}, p.41, and this number
  is \( p(n)-g(n,p) \).
 
 Finally we note that \( \{\overline{B}_q | q \in A \} \), a set of 
 nilpotent elements, is contained in \( {\cal R}(n,p) \) and is linearly 
 independent of the subspace $U$.  Therefore
 \begin{eqnarray*}
 \dim {\cal R}(n,p) & \geq & 2^{n-1} -p(n)+p(n)-g(n,p)\\
  &=& 2^{n-1}-g(n,p)\\
  &=& \dim \Sigma(n,p) - \dim G(n,p)\\
  &=& \dim \ker \phi
  \end{eqnarray*}
  \end{proof} 
  
 We can now describe \( {\cal R}(n,p) \) exactly.
    
 \begin{thm}
 ${\cal R}(n,p)=\ker\phi$ and is spanned by all
 \( \overline{B}_{q}-\overline{B}_{r} \) with \(  q\approx  r \) together 
 with all \(\overline{B}_{q} \) where $q$\ has a component of
 multiplicity p or more.
 \label{ba-sis}
 \end{thm}
 \begin{proof} Lemma~\ref{ra-dinker} and Lemma~\ref{di-mr} prove that 
 ${\cal R}(n,p)=\ker\phi$.  The proof of Lemma~\ref{di-mr} then shows that
  ${\cal R}(n,p)$ not only contains but is actually spanned by  all
 \( \overline{B}_{q}-\overline{B}_{r} \) with \(  q\approx  r \) together 
 with all \(\overline{B}_{q} \) where $q$\ has a component of
 multiplicity p or more.
 \end{proof}
 
From Theorem~\ref{ba-sis} it follows that \( \dim\Sigma(n,p)/{\cal 
R}(n,p)=g(n,p) \) and so, by Theorem~\ref{co-mutt}, \( \Sigma(n,p) \) has \( 
g(n,p) \) irreducible representations all of which are $1$-dimensional.  
We may describe them as follows.

Let \( \pi \) be any partition of $n$ and $x$ any element of  \( 
\Sigma(n,p) \).  Then \( \phi(x) \) is a $p$-modular character of \( S_n 
\) and we let \( \phi(x)^{\pi} \) be the value of this character on the 
conjugacy class corresponding to \( \pi \).  Define \( 
\lambda_{\pi}:\Sigma(n,p)\rightarrow{\cal F}_p \) by
\[ \lambda_{\pi}(x)=\phi(x)^{\pi}\mbox{ for all }x\in\Sigma(n,p) \]
It follows, since \( \phi \) is a homomorphism and characters of \(S_n\) 
are added and multiplied pointwise, that \( \lambda_{\pi} \) is a 
($1$-dimensional) representation of \( \Sigma(n,p) \).

\( \lambda_{\pi} \) is determined by its values \( 
\phi(\overline{B}_q^{\pi})=\tilde{\chi}^{\pi} \) on the basis of \( 
\Sigma(n,p) \) and, by ordering the basis, we can define a column vector 
\( D^{\pi} \) of these values.  By the proof of Lemma~\ref{di-gnp} the matrix 
whose columns are the vectors \( D^{\pi} \) has rank \( g(n,p) \).  That 
lemma also shows that   the set of $p$-regular partitions  provides a 
suitable set of distinct columns that may be taken to define \( g(n,p)\) 
distinct irreducible representations of \( \Sigma(n,p) \).

\section{The Nilpotency Index of the Radical} 
 
Let \( Y_m \) be the subspace of \( \Sigma(n,p) \) spanned by all \( 
\overline{B}_q \) where $q$ has $m$ or more components (for simplicity of 
notation we omit the reference to the dependency on $n$ and $p$).  Then
\[ \Sigma(n,p) = Y_1 \supseteq Y_2 \supseteq \ldots \supseteq Y_n \supseteq 
Y_{n+1}=0 \]

\begin{lem} \( Y_m {\cal R}(n,p) \subseteq Y_{m+1} \)
\label{ym-r}
\end{lem}
\begin{proof} 
 Let $s$ be a composition with at least $m$ components (so that 
 \( \overline{B}_s\in Y_m \))
  and consider the product \( \overline{B}_sX \) for each of the 
 spanning elements of ${\cal R}(n,p)$ given in Theorem~\ref{ba-sis}.  
 Such a product is, by Lemma \ref {co-effbqbr},
  a linear combination of terms \( \overline{B}_t \) with 
 \( t\preceq s \) but, as we now prove, the term \( \overline{B}_s \) itself 
 occurs with coefficient zero.  There are two cases to consider:
 \begin{enumerate}
 \item \( X=\overline{B}_q-\overline{B}_r, q\approx r \).  By 
 Lemma~\ref{co-effbqbr}, the coefficients of \( \overline{B}_s \) in both 
 \( \overline{B}_s\overline{B}_q \) and \( \overline{B}_s\overline{B}_r 
 \) are equal; thus, in \( \overline{B}_s(\overline{B}_q-\overline{B}_r) \), 
 the coefficient of \( \overline{B}_s \) is zero.
 \item \( X=\overline{B}_r \) where $r$ has \( t_i \) components equal to 
 $i$ with at least one \( t_i \) being $p$ or more.  Again, by
  Lemma~\ref{co-effbqbr} since \( t_1!\ldots  t_n! \) is zero in \( {\cal 
  F}_p \), the coefficient of \( \overline{B}_s \) in \( 
  \overline{B}_s\overline{B}_r \) is zero.
  \end{enumerate}
It now follows that \( Y_m X\subseteq Y_{m+1} \) for all \( X\in {\cal 
R}(n,p) \) and this completes the proof.
\end{proof}

Let \( {\cal T} \) denote the subspace of \( {\cal R}(n,p) \) generated by all 
\( \overline{B}_q-\overline{B}_r \) with \( q\approx r \) (again we omit 
the reference to the dependency on $n$ and $p$).  Since \( {\cal T} \)
is the image of \(R_n\) under the homomorphism \( \rho_1 \), \( {\cal T} \) 
is a nilpotent ideal and therefore is contained in \( {\cal R}(n,p) \).

\begin{lem}
\begin{enumerate}
\item If $n$ is odd or \( p \neq 2 \) then \( {\cal R}(n,p) \subseteq 
Y_2\cap {\cal T}+Y_3 \)
\item If $n$ is even and \( p=2 \) then \( {\cal R}(n,p) \subseteq 
\langle \overline{B}_{[n/2,n/2]}\rangle+
Y_2\cap{\cal T}+Y_3 \)
\end{enumerate}
\label{tw-cases}
\end{lem}
\begin{proof}
Consider the spanning set for  \( {\cal R}(n,p) \) given in 
Theorem~\ref{ba-sis}.  An element \( \overline{B}_q-\overline{B}_r \) with 
\( q\approx r \) is non-zero only if $q$ and $r$ have at least $2$ 
components and so such an element belongs to \(Y_2\cap{\cal T} \).

Consider an element \( \overline{B}_q \) where the composition $q$ has a 
component which occurs $p$ times or more.  If $n$ is odd or \( p \neq 2 
\) then $q$ will have at least $3$ components and so \( \overline{B}_q 
\in Y_3 \).  The composition $q$ can have fewer than $3$ components only 
if $p=2$ and \( q=[n/2,n/2] \).  The lemma now follows.
\end{proof}

\begin{lem} If $n$ is even and \( p=2 \) then
\[{\cal R}(n,p)^2 \subseteq Y_3\cap{\cal T}+Y_4 \]
\label{on-case}
\end{lem}
\begin{proof}
By Lemma~\ref{ym-r} and Lemma~\ref{tw-cases} 
\[ {\cal R}(n,p)^2 \subseteq 
\langle \overline{B}_{[n/2,n/2]}\rangle {\cal R}(n,p)+
Y_3\cap{\cal T}+Y_4 \]
and so it is sufficient to prove that all products
\(\overline{B}_{[n/2,n/2]}X\) lie in \(Y_3\cap{\cal T}+Y_4 \)
where $X$ runs through the spanning set of \({\cal R}(n,p) \) given in 
Theorem~\ref{ba-sis}.  If \( X\in {\cal T} \) then, as \( 
\overline{B}_{[n/2,n/2]}\in Y_2 \) and \( {\cal T} \) is a two-sided ideal, \( 
\overline{B}_{[n/2,n/2]}X\in Y_3\cap {\cal T} \).

Suppose that \( X=\overline{B}_q\) where \(q=[a_1,\ldots,a_r]\) has a 
repeated part. Then \(\overline{B}_{[n/2,n/2]}X\) is a sum of elements 
\(\overline{B}_s\), one for each \(2 \times r \) matrix $Z$ in \(S([n/2,n/2],q)\).  
If such a 
matrix $Z$ has $4$ or more non-zero entries then it contributes a summand 
\(\overline{B}_s \in Y_4 \).  If it has $3$  non-zero entries then its two 
rows will not be equal and it may be paired with the matrix \( \bar{Z} \) 
obtained from \( Z \) by interchanging the rows.  This pair of matrices 
contributes a summand \( \overline{B}_u+\overline{B}_v \) with \( u\approx v \) which lies in 
\(Y_3\cap{\cal T} \).  Finally, if $Z$ has 2 non-zero entries only it 
will have 
one of two possible forms each of which contributes a summand
 \( \overline{B}_{[n/2,n/2]}\); since \( p=2 \) this contribution is zero.
\end{proof}

We can now give the main result of this section.

 \begin{thm}
 If \( n\geq 3 \)  the nilpotency index of ${\cal R}(n,p)$ is
  \( n-1 \).  
 \end{thm}

\begin{proof}  In the proof of Corollary 3.5 of  \cite{Atkinson} it was proved that, 
 if \( w=B_{[1,n-1]}-B_{[n-1,1]} \) and \( D(a,b)=B_{[1^a,n-a-b,1^b]} \) then 
 \[ w^{r}=\sum^{r}_{k=0} (-1)^{k}{r\choose k} D(r-k,k) \]
 In particular, \( w^{n-2} \not\in {\cal P}_n \) so that \( x=\rho_1(w) 
 \) is an element of \( {\cal R}(n,p) \) and \( x^{n-2}\not =0 \).  
 Therefore the nilpotency index of \( {\cal R}(n,p) \) is not less than 
 \( n-1 \).

To prove that the nilpotency index is no more than $n-1$ we consider two 
cases.  First, suppose that either $n$ is odd or \( p \neq 2 \).  Then
 Lemma~\ref{ym-r} and
Lemma~\ref{tw-cases} show that
\begin{eqnarray*}
{\cal R}(n,p)^{n-1} & \subseteq & (Y_2 \cap {\cal T}){\cal R}(n,p)^{n-2}+
Y_3{\cal R}(n,p)^{n-2} \\
& \subseteq & Y_n \cap {\cal T} + Y_{n+1}
\end{eqnarray*}
On the other hand, if $n$ is even and \( p=2 \), Lemma~\ref{ym-r} and 
Lemma~\ref{on-case} show that
\begin{eqnarray*}
{\cal R}(n,p)^{n-1}&=&{\cal R}(n,p)^2{\cal R}(n,p)^{n-3}\\
&\subseteq & (Y_3\cap{\cal T}){\cal R}(n,p)^{n-3}+Y_4{\cal R}(n,p)^{n-3}\\
&\subseteq & Y_n \cap {\cal T} + Y_{n+1}
\end{eqnarray*}
 However, since \( Y_{n+1}=0 \) and \( Y_n \cap {\cal T} =0 \), the 
result now follows.
\end{proof}

\begin{rem}  By direct calculation we see that \( {\cal R}(1,p)=0 \) and 
that \( {\cal R}(2,p) = \langle\overline{B}_{[1,1]}\rangle \) (so has nilpotency index $2$).
\end{rem}

\affiliationone{Mathematical Institute\\ University of St. Andrews\\ North Haugh\\ St Andrews\\
Fife KY16 9SS, UK} 

\end{document}